\documentclass{amsart}
\usepackage{graphicx}
\usepackage{amsfonts}
\usepackage{amsthm}
\usepackage{amsmath}
\usepackage{enumerate}
\usepackage{a4wide}

\newcommand{\mt}{\mu_t}
\newcommand{\mx}{\mu_x}

\renewcommand{\d}{\mathrm{d}}

\newcommand{\ddt}{{\Delta_t}}

\newcommand{\prn}[1]{\left( #1 \right)}                 
\newcommand{\bra}[1]{\left[ #1 \right]}
\newcommand{\brc}[1]{\left\lbrace #1 \right\rbrace}

\newcommand{\Int}{\mathbb{Z}}
\newcommand{\Nat}{\mathbb{N}}
\newcommand{\Natn}{\mathbb{N}_0}
\newcommand{\Real}{\mathbb{R}}
\newcommand{\TS}{\mathbb{T}}

\newcommand{\To}{\rightarrow}

\newtheorem{thm}{thm}

\newtheorem{lem}[thm]{lem}
\newtheorem{cor}[thm]{Corollary}

\theoremstyle{def}

\newtheorem{rem}[thm]{Remark}
\newtheorem{exmp}[thm]{exmp}

\begin{document}


\title{Transport Equation on Semidiscrete Domains and Poisson-Bernoulli Processes}\footnote{This is a preprint. The final version of this paper will appear in Journal of Difference Equations and Applications}
\author{Petr Stehl\'{\i}k, Jon\'{a}\v{s} Volek}

\maketitle
\begin{center}
\footnotesize
\begin{minipage}[c]{11cm}
\subsection*{Abstract}In this paper we consider a scalar transport equation with constant coefficients on domains with discrete space and continuous, discrete or general time. We show that on all these underlying domains solutions of the transport equation can conserve sign and integrals both in time and space. Detailed analysis reveals that, under some initial conditions, the solutions correspond to counting stochastic processes and related probability distributions. Consequently, the transport equation could generate various modifications of these processes and distributions and provide some insights into corresponding convergence questions. Possible applications are suggested and discussed.
\\

\subsection*{Keywords} Transport Equation, Advection Equation, Conservation Law, Time Scales, Semidiscrete Method, Poisson process, Bernoulli process.
\\
\subsection*{AMS Subject classification} 34N05, 35F10, 39A14, 65M06.
\end{minipage}
\end{center}

\section{Introduction}
Scalar transport equation with constant coefficients $u_t+k u_x=0$ belongs among the simplest partial differential equations. Its importance is based on the following facts. Firstly, it describes advective transport of fluids, as well as one-way wave propagation. Secondly, it serves as a base for a study of hyperbolic partial differential equations (and is consequently analysed also in numerical analysis). Thirdly, its nonlinear modifications model complex transport of fluids, heat or mass. Finally, its study is closely connected to conservation laws (see \cite{bEva} or \cite{bLev}).

Properties  and solutions of partial difference equations have been studied mainly from numerical (e.g. \cite{bLev}) but also from analytical point of view (e.g. \cite{bChe}). Meanwhile, in one dimension, there has been a wide interest in the problems with mixed timing, which has been recently clustered around the time scales calculus and the so-called dynamic equations (see \cite{bBP}, \cite{aHil}). Nevertheless, there is only limited literature on partial equations on time scales (see \cite{aAM}, \cite{aBG}, \cite{a_maxpripde}). These papers indicate the complexity of such settings and the necessity to analyze basic problems like transport equation. Our analysis is also closely related to numerical semidiscrete methods (e.g. \cite[Section 10.4]{bLev}) or analytical Rothe method (e.g. \cite{aRot}).

In this paper we consider a transport equation on domains with discrete space and general (continuous, discrete and time scale) time (see Figure \ref{f:domains}). We show that the solutions of transport equation does not propagate along characteristics lines as in the classical case and feature behavior close to the classical diffusion equation. Our analysis of sign and integral conservation discloses interesting relationship between the solutions on such domains and probability distributions related to Poisson and Bernoulli stochastic processes. These counting processes are used to model waiting times for occurence of certain events (defects, phone calls, customers' arrivals, etc.), see \cite{bBer}, \cite{bGha} or \cite{bPan} for more details. Consequently, considering domains with general time, we are able not only to generalize these standard processes but also generate transitional processes of Poisson-Bernoulli type and corresponding distributions. Moreover, our analysis provides a different perspective on some numerical questions (numerical diffusion) and relate it to analytical problems (relationship between the CFL condition and regressivity). Finally, it also establishes relationship between the time scales calculus and heterogeneous and mixed probability distributions in the probability theory.

In Section \ref{s:classical} we summarize well-known features of the classical transport equation. In Section \ref{s:con-time} we consider a transport equation with discrete space and continuous time. In Section \ref{s:dis-time} we solve the problem on domains with discrete time. In Section \ref{s:ts-time} we generalize those results to domains with a general time and prove the necessary and sufficient conditions which ensure that the sign and both time and space integrals are conserved (Theorem \ref{t:ts-pdf}). Finally, in Section \ref{s:processes} we discuss convergence issues, applications to probability distributions and stochastic processes and provide two examples.

\begin{figure}[t]
\vspace{42pt}
\begin{center}
\includegraphics[height=1.7in,width=1.7in]{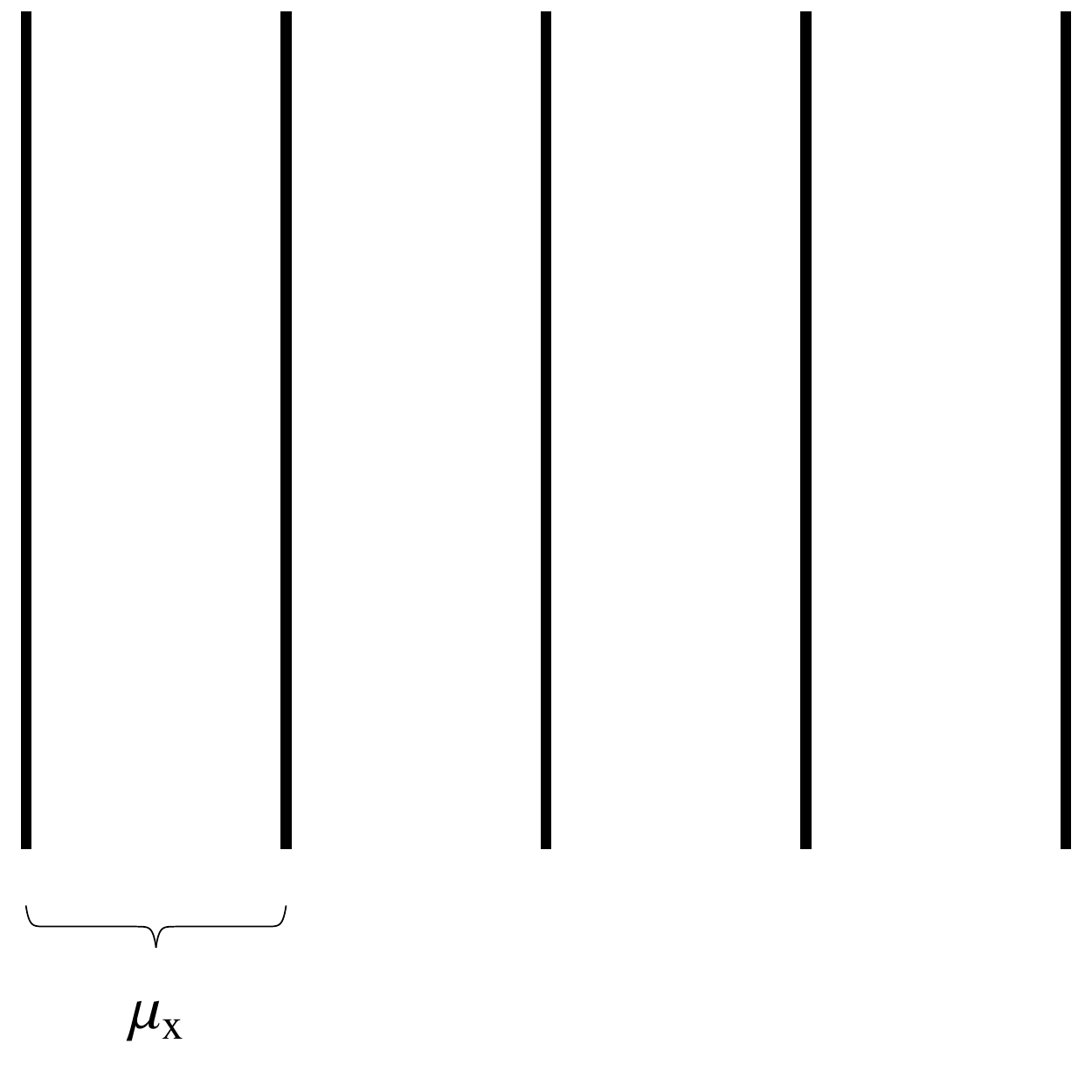} \hspace{0.2in} \ 
\includegraphics[height=1.7in,width=1.7in]{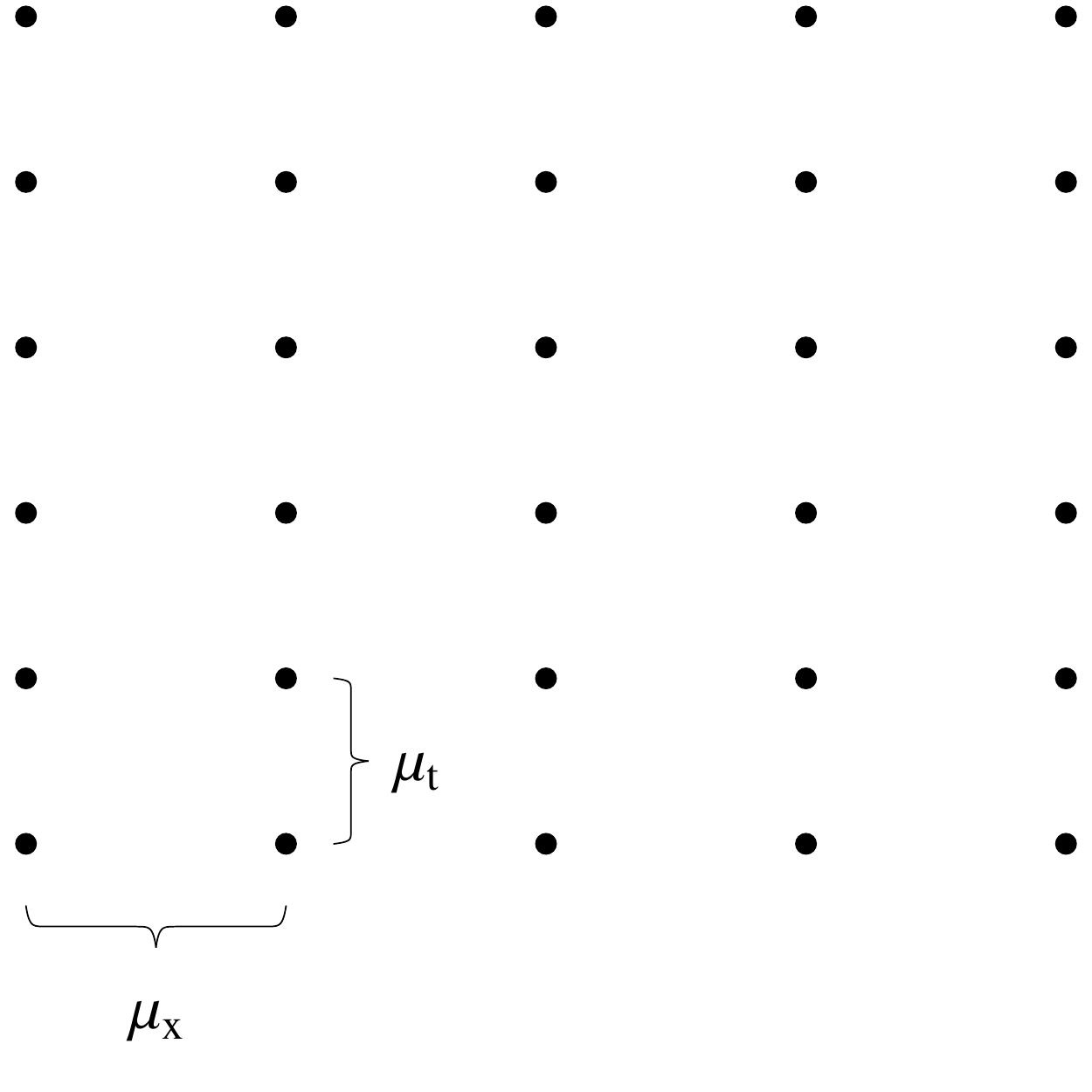} \hspace{0.2in} \
\includegraphics[height=1.7in,width=1.7in]{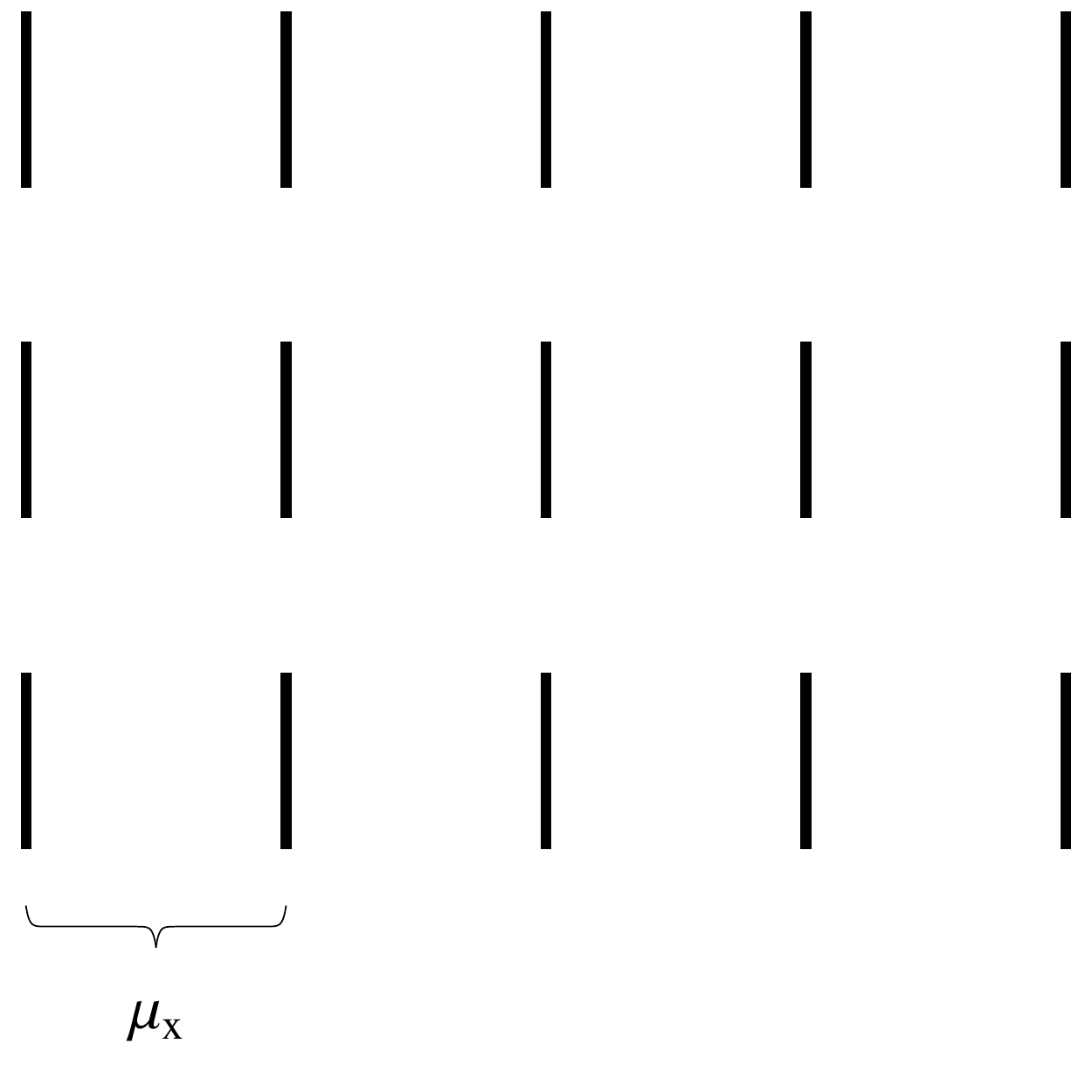}
\end{center}

\caption{Examples of various domains considered in this paper. We study domains with discrete space and continuous (Section \ref{s:con-time}), discrete (Section \ref{s:dis-time}) and general time (Section \ref{s:ts-time}).}
\label{f:domains}
\end{figure}


\section{Preliminaries and Notation}\label{s:prel}
The sets $\Real$, $\Int$, $\Nat$ denote real, integer and natural numbers. Furthermore, let us introduce $\Natn=\Nat\cup\brc{0}$ and $\Real_0^+=[0,\infty)$. Finally, we use multiples of discrete number sets, e.g. $a$-multiple of integers is denoted by $a\Int$ and defined by $a\Int=\brc{\ldots,-2a,-a,0,a,2a,\ldots}$.

Partial derivatives are denoted by $u_t(x,t)$ and $u_x(x,t)$ and partial differences by
\begin{equation}\label{e:nabla}
\Delta_t u(x,t) = \frac{u(x,t+\mt)-u(x,t)}{\mt}\quad \mathrm{and} \quad \nabla_x u(x,t) =\frac{u(x,t)-u(x-\mx,t)}{\mx},
\end{equation}
where $\mt$ and $\mx$ denote step sizes in time and space.

In Section \ref{s:ts-time}, we consider time to be a general time scale $\TS$, i.e. an arbitrary closed subset of $\Real$. Time step could be variable, described by a graininess function $\mt:\TS\To\Real_0^+$. We use the partial delta derivative $u^\ddt$ which reduces to $u_t$ in points in which $\mt(t)=0$ or to $\Delta_t u$ in those $t$ in which $\mu(t)>0$. Similarly, we work with the so-called delta-integral which corresponds to standard integration if $\TS=\Real$ or to summation if $\TS=\Int$. Finally, the dynamic exponential function $e_p(x,x_0)$ is defined as a solution of the initial value problem (under the regressivity condition $1+p(t)\mu(t) \neq 0$)
\[
	\begin{cases}
		x^\Delta(t)=p(t) x(t),\\
		x(0)=1.
	\end{cases}
\] 
For more details concerning the time scale calculus we refer the interested reader to the survey monograph \cite{bBP}.

Given function $u(x,t)$, by $u(x,\cdot)$ we mean functions of one variable having the form $u(0,t)$, $u(1,t)$, etc. Similarly, by $u(\cdot,t)$ we understand one-dimensional sections of $u(x,t)$ having the form $u(x,0)$, $u(x,1)$, etc.

\section{Classical Transport Equation}\label{s:classical}
Let us briefly summarize essential properties of the classical transport equation
\begin{equation}\label{e:con-space}
\begin{cases}
u_t(x,t)+k u_x(x,t) = 0, \quad t\in\Real_0^+, x\in\Real,\\
u(x,0)=\phi(x),  \quad  x\in\Real.
\end{cases}
\end{equation}
with $\phi\in C^1$. Typical features whose counterparts are studied in this paper include:

\begin{itemize}
	\item the unique solution $u(x,t)=\phi(x-kt)$ could be obtained via the \emph{method of characteristics}, the solution is constant on the \emph{characteristic lines} where $x-kt=C$,
	\item consequently, the solution \emph{conserves sign}, i.e. if $\phi(x)\geq 0$ then $u(x,t)\geq 0$,
	\item moreover, the solution \emph{conserves integral} in space sections, i.e. if $\int\limits_{-\infty}^\infty \phi(x) \d x = K$, then
	\[
		\int\limits_{-\infty}^\infty u(x,t) \d x = K,\quad \mathrm{for\ all\ } t\geq 0,
	\]
	\item finally, the solution \emph{conserves integral} in time sections in the following sense. For $k>0$ we have that
	\[
		\int\limits_{0}^\infty u(x,t) \d t = \frac{1}{k} \int\limits_{-\infty}^x \phi(s) \d s.
	\]
	Consequently, if $\phi(x)=0$ for $x\geq x_0$, then the integral along time sections is constant for all $x\geq x_0$.
\end{itemize}

\section{Discrete Space and Continuous Time}\label{s:con-time}
In contrast to the classical problem \eqref{e:con-space} we consider a domain with discrete space and the problem
\begin{equation}\label{e:con-time}
\begin{cases}
u_t(x,t)+k\nabla_x u(x,t) = 0, \quad t\in\Real_0^+, x\in\Int,\\
u(x,0)=\begin{cases}
	A, \quad x=0,\\
	0, \quad x\neq 0,
	\end{cases}
\end{cases}
\end{equation}
where $A>0$, $k>0$ and $\nabla_x u$ reduces to\footnote{We assume that $k>0$ so that the solution is bounded and does not vanish. Moreover, we use the nabla difference instead of delta difference. The single reason is the simpler form of the solution \eqref{e:con:solution}. If we used the delta difference, we would consider $k<0$ and the solution would propagate to the quadrant with $t>0$ and $x<0$. This applies also to the problems which we study in the following sections.}
\[
\nabla_x u(x,t) = u(x,t)-u(x-1,t).
\]
One could rewrite the equation in \eqref{e:con-time} into
\[
u_t(x,t)=-k u(x,t) + k u(x-1,t),
\]
which implies that the problem  \eqref{e:con-time} could be viewed as an infinite system of differential equations.

\begin{lem}\label{l:con-solution}
The unique solution of the problem \eqref{e:con-time} has the form:
\begin{equation} \label{e:con:solution}
		\displaystyle u(x, t) = \left\lbrace
		\begin{array}{l l}
			\displaystyle A \frac{k^x}{x!}t^x e^{-kt}, & \quad t\in\Real_0^+, x\in\Natn,\\
			&\\
			0, & \quad t\in\Real_0^+, x\in\Int, x<0.
		\end{array}	
		\right.
	\end{equation}
\end{lem}
\begin{proof}
First, let us observe that $u(x,t)=0$ for all $t\in\Real_0^+$, $x<0$. The uniqueness of the trivial solution for $x<0$ follows e.g. from \cite[Corollary 1]{aRei} or more generally from \cite[Theorem 3.1.3]{bCZ}. Let us prove the rest (i.e. $x\geq 0$ by mathematical induction. Obviously, we have that $u(0,t)=Ae^{-kt}$, since $u_t(0,t)=-ku(0,t)+ ku(-1,t)=-ku(0,t)$ and $u(0,0)=A$.\par
Moreover, if we assume that $x\in\Nat_0$ and $u(x,t)= A \frac{k^x}{x!}t^x e^{-kt}$, then $u(x+1,t)$ satisfies
\[
\begin{cases}
u_t(x+1,t)=-k u(x+1,t) + A \frac{k^x}{x!}t^x e^{-kt} , \quad t\in\Real_0^+,\\
u(x,0)=0.
\end{cases}
\]
One could use the variation of parameters to show that  the unique solution is $u(x+1,t)=A \frac{k^{x+1}}{(x+1)!}t^{x+1} e^{-kt}$, which proves the inductive step and consequently finishes the proof.
\end{proof}

Let us analyze the sign and integral preservation of \eqref{e:con-time}.

\begin{lem}
The solution of the problem \eqref{e:con-time} conserves the sign, the integral in time and the sum in the space variable.
\end{lem}
\begin{proof}
The sign preservation follows from the positivity of all terms in \eqref{e:con:solution}. Next, we could use integration by parts to obtain (we skip the details since we prove this result in more general settings in Theorem \ref{t:ts-integral-2})
\begin{equation}\label{e:con:Erlang}
\int_{0}^\infty u(x,t) \d t = \frac{A}{k}.
\end{equation}
Similarly, summing over $x$ we get
\begin{equation}\label{e:con:Poisson}
\sum\limits_{x=0}^{\infty} A \frac{k^x}{x!}t^x e^{-kt} = A e^{-kt} \sum\limits_{x=0}^{\infty} \frac{(kt)^x}{x!} = A e^{-kt} e^{kt}=  A.
\end{equation}
\end{proof}

\begin{figure}[t]
\vspace{42pt}
\begin{center}
\includegraphics[height=1.7in,width=1.7in]{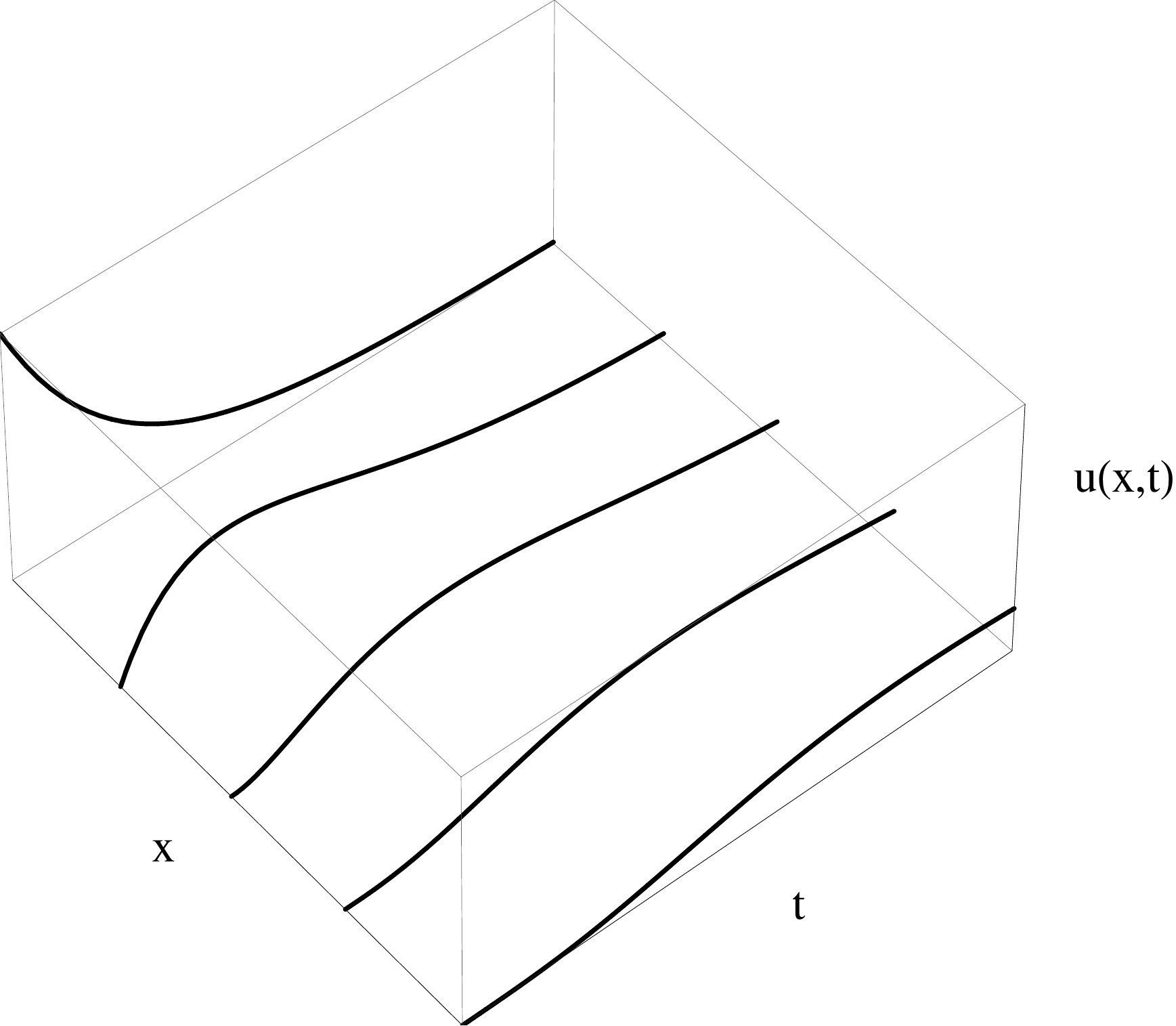} \hspace{0.2in} \ 
\includegraphics[height=1.7in,width=1.7in]{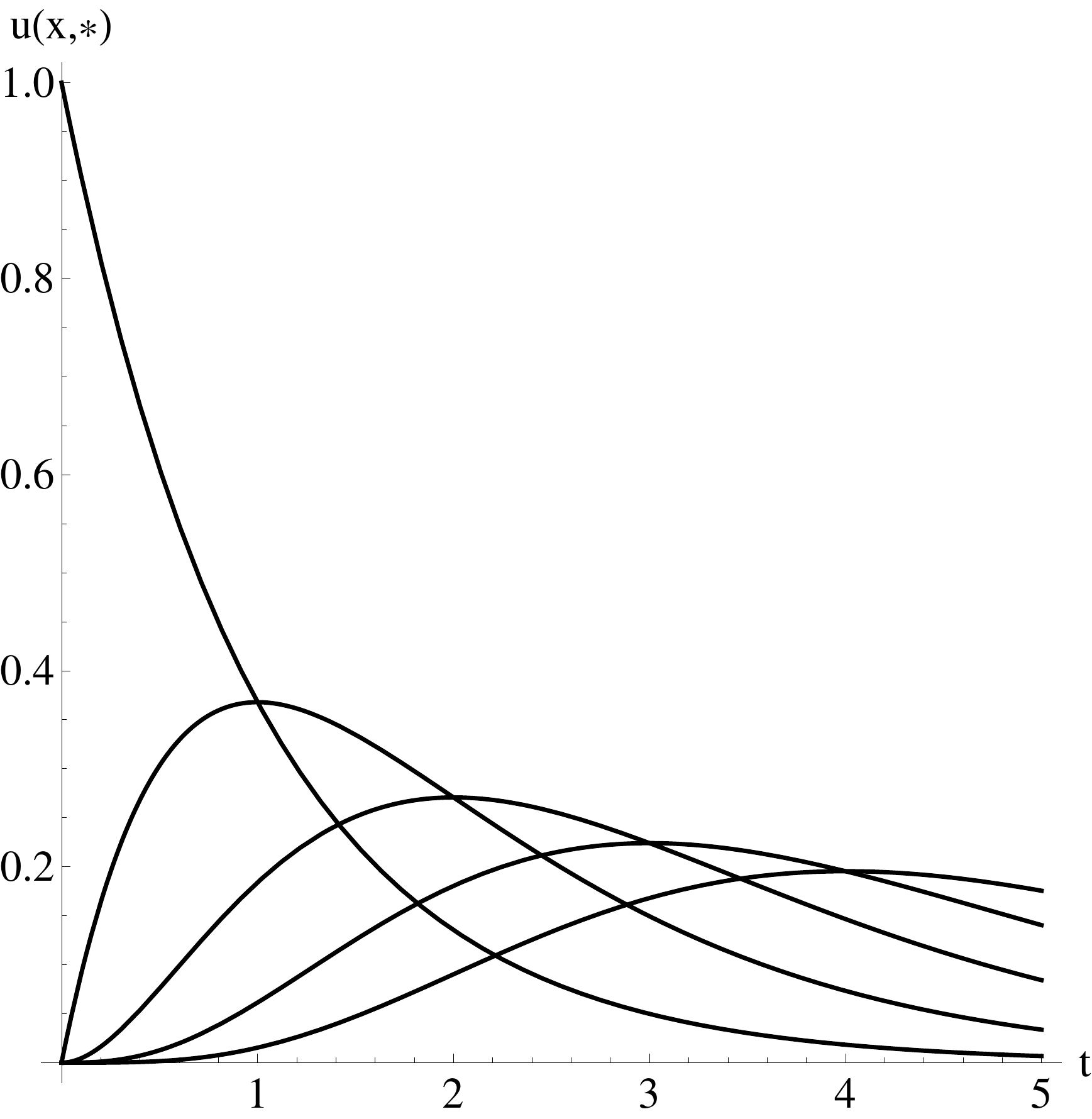} \hspace{0.2in} \
\includegraphics[height=1.7in,width=1.7in]{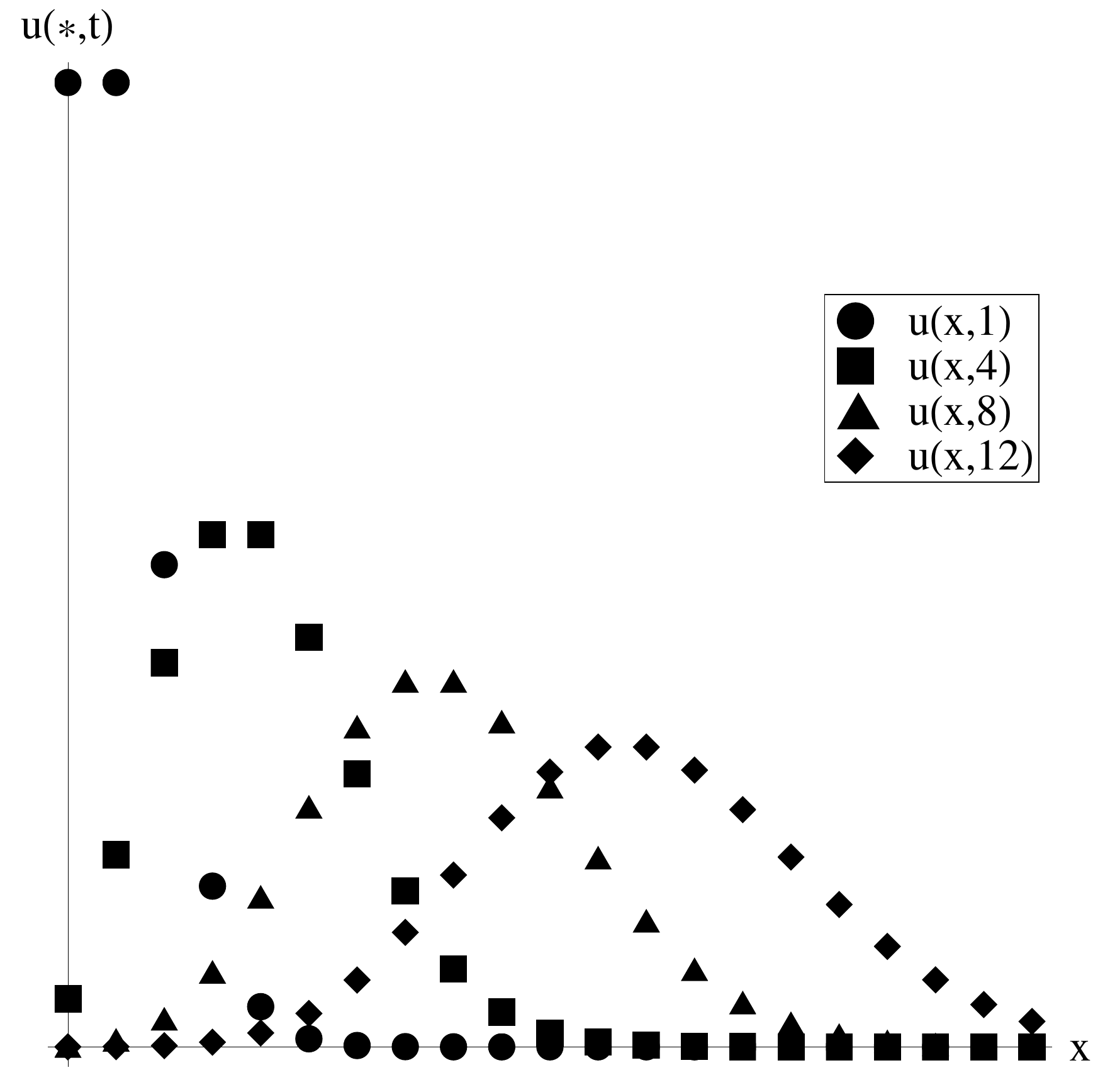}
\end{center}

\caption{Solution of the transport equation with discrete space and continuous time \eqref{e:con-time} with $A=1$ and $k=1$.}
\label{f:con-time}

\end{figure}


If we go deeper and analyze values obtained in \eqref{e:con:Erlang} and \eqref{e:con:Poisson} we get the first indication of the relationship of the semidiscrete transport equation with stochastic processes.

\begin{rem}\label{r:con-poisson}
If $A=k$ then time sections of the solution \eqref{e:con:solution} generate the probability density function of Erlang distributions (note that for $x=0$ we get the exponential distribution and that Erlang distributions are special cases of Gamma distributions).

Similarly, if $A=1$ the space sections of \eqref{e:con:solution} form the probability mass functions of  Poisson distributions.

Consequently, if $A=k=1$ the solution $u(x,t)$ describes Poisson process. All these facts are further discussed in Section \ref{s:processes}.
\end{rem}

We conclude this section with two natural extensions. Firstly, we mention possible generalizations to other discrete space structures.

\begin{rem}\label{r:con-mx}
If we consider the problem \eqref{e:con-time} on a domain with a discrete space having the constant step $\mx>0$, not necessarily $\mx=1$, we obtain qualitatively equivalent problem, since
\begin{align*}
u_t(x,t)+k \frac{u(x,t)-u(x-\mx,t)}{\mx} &=u_t(x,t)+ \frac{k}{\mx} \prn{u(x,t)-u(x-\mx,t)}\\
&=u_t(x,t)+\hat{k}\nabla_x u(x,t).
\end{align*}
In contrast to the rest of this paper, the value of $\mx$ does not play essential role here. Therefore, for presentation purposes, we restricted our attention to $\mx=1$.
\end{rem}

Finally, we discuss more general initial condition and show that the solution is the sum of point initial conditions which justifies their use not only in this section but also in the remainder of this paper.

\begin{cor}\label{c:con-general-ic}
The unique solution of
\begin{equation}\label{e:con-time-gen-ic}
\begin{cases}
u_t(x,t)+k\nabla_x u(x,t) = 0, \quad t\in\Real_0^+, x\in\Int,\\
u(x,0)=C_x.
\end{cases}
\end{equation}
is given by
\begin{equation}\label{e:con-time-gen-ic-sol}
u(x,t)=\sum\limits_{i=-\infty}^{x} C_{i} \frac{(kt)^{x-i}}{(x-i)!} e^{-kt}.
\end{equation}
\end{cor}
\begin{proof}
One could split \eqref{e:con-time-gen-ic} into problems with point initial conditions, use Lemma \ref{l:con-solution} to solve them and then employ linearity of the equation to get \eqref{e:con-time-gen-ic-sol}.
\end{proof}

\section{Discrete Space and Discrete Time}\label{s:dis-time}
In this section, we assume that both time and space are homogenously discrete with steps $\mt>0$ and $\mx>0$ respectively. In other words, we consider a discrete domain
\[
\Omega=\brc{(x,t)=\prn{m\mx,n\mt},\ \textrm{with } m\in\Int, n\in\Natn}.
\]
The transport equation and the corresponding problem then have the form

\begin{equation}\label{e:dis-time}
\begin{cases}
\Delta_t u(x,t)+k\nabla_x u(x,t) = 0, \quad (x,t)\in\Omega,\\
u(x,0)=\begin{cases}
A, \quad x=0,\\
0, \quad x\neq 0,
\end{cases}
\end{cases}
\end{equation}
where $A>0$, $k>0$. Using the definition of partial differences in \eqref{e:nabla}, we can easily rewrite the equation in \eqref{e:dis-time} into
\begin{equation}\label{e:dis-time-2}
u(x,t+\mt)=\prn{1-\frac{k\mt}{\mx}}u(x,t)+\frac{k\mt}{\mx}u(x-\mx,t)
\end{equation}
and derive the unique solution.
\begin{lem}\label{l:dis-solution}
Let $m\in\Int$ and $n\in\Nat_0$. The unique solution of \eqref{e:dis-time} has the form:
\begin{equation} \label{e:dis:solution}
		\displaystyle u(m\mu_{x}, n\mu_{t}) = \left\lbrace
		\begin{array}{l l}
			\displaystyle A \binom{n}{m} \left( 1 - \frac{k\mu_{t}}{\mu_{x}} \right)^{n-m} \left( \frac{k\mu_{t}}					{\mu_{x}} \right)^{m}, & \quad n \geq m \geq 0,\\
			&\\
			0, & \quad 0\leq n < m,\textrm{ or } m<0.
		\end{array}	
		\right.
	\end{equation}
\end{lem}
\begin{proof}
First, let us show that the solution vanishes uniquely for $u(-m\mx,n\mx)= 0$,  for all $m, n\in\Nat$. Consulting \eqref{e:dis-time-2}, we observe that the value of $u(-m\mx,n\mx)$ is obtained as a linear combination of initial conditions $u(-m\mx,0)$, $u(-(m+1)\mx,0)$, \ldots, $u(-(m+n)\mx,0)$, i.e. a linear combination of $n+1$ zeros.

We prove the rest of the statement by induction. Apparently,
\[
u(0,n\mu_{t}) = \prn{1-\frac{k\mt}{\mx}}^{n} u(0,0) = A \prn{1-\frac{k\mt}{\mx}}^{n}.
\]
Next, let us assume that $u(m\mu_{x}, n\mu_{t})$ satisfies \eqref{e:dis:solution}, then
\begin{displaymath}
			\begin{array}{l}
					\displaystyle u \left( (m+1)\mu_{x},n\mu_{t} \right)  =  \displaystyle \prn{1-\frac{k\mt}{\mx}}^{n} 												u((m+1)\mu_{x}, 0) \\
\quad + \sum \limits_{r_{m+1}=0}^{n-1}  \prn{1-\frac{k\mt}{\mx}}^{n-1-r_{m+1}} A \prn{1-\frac{k\mt}{\mx}}^{r_{m+1}-m} \prn{\frac{k\mt}{\mx}}^{m+1} \sum 							\limits_{r_{m}=0}^{r_{m+1}-1} \ldots \sum \limits_{r_{2}=0}^{r_{3}-1} \sum 												\limits_{r_{1}=0}^{r_{2}-1} 1\\

= A \prn{1-\frac{k\mt}{\mx}}^{n-(m+1)} \prn{\frac{k\mt}{\mx}}^{m+1} \sum \limits_{r_{m+1}=0}^{n-1} \ldots \sum 										\limits_{r_{2}=0}^{r_{3}-1} \sum \limits_{r_{1}=0}^{r_{2}-1} 1.
				\end{array}
\end{displaymath}

At this stage, let us observe the properties of the falling factorials (see e.g. \cite[Section 2.1]{bEla} or \cite[Section 2.1]{bKP}) to get that
\[
\sum \limits_{r_{m+1}=0}^{n-1} \ldots \sum \limits_{r_{2}=0}^{r_{3}-1} \sum \limits_{r_{1}=0}^{r_{2}-1} 1=\frac{n^{\underline{m+1}}}{(m+1)!} = {n \choose m+1},
\]
which finishes the proof.
\end{proof}
The closed-form solution enables us to analyze sign and integral conservation.
\begin{lem}\label{l:dis-preservation}
If the inequality
\begin{equation}\label{e:D1}
1-\frac{k\mt}{\mx}>0, \tag{D1}
\end{equation}
holds then the solution of \eqref{e:dis-time} satisfies
\begin{enumerate}[\upshape (i)]
	\item $u(x,t)\geq 0$,	
	\item $\sum\limits_{m=-\infty}^{\infty} u(m\mx,t)$ is constant for all $t=\brc{0,\mt,2\mt,\ldots}$,
	\item $\sum\limits_{n=0}^{\infty} u(x,n\mt)$ is constant for all $x=\brc{0,\mx,2\mx,\ldots}$.
\end{enumerate}
\end{lem}

\begin{proof}
\begin{enumerate}[\upshape (i)]
	\item The inequality follows immediately from Lemma \ref{l:dis-solution} .
	\item If we fix $t$ and sum up the equation  \eqref{e:dis-time-2} over $x$ we get
		\[
			\sum\limits_{m=-\infty}^{\infty} u(m\mx,t+\mt)= \prn{1-\frac{k\mt}{\mx}} \sum\limits_{m=-\infty}^{\infty} u(m\mx,t)+\frac{k\mt}{\mx} \sum\limits_{m=-\infty}^{\infty}u((m-1)\mx,t).
		\]
		The assumption (D1) implies that the sum on the left hand side is a linear combination of two sums on the right hand side. Since these sums are equal, we get that 
		\[
		 \sum\limits_{m=-\infty}^{\infty} u(m\mx,t+\mt)= \sum\limits_{m=-\infty}^{\infty} u(m\mx,t).
		\]
	\item Similarly, one could sum up the equation  \eqref{e:dis-time-2} over $t$ to get for a fixed $x>0$
		\[
			\sum\limits_{n=1}^{\infty} u(x,n\mt)= \prn{1-\frac{k\mt}{\mx}} \sum\limits_{n=0}^{\infty} u(x,n\mt)+\frac{k\mt}{\mx} \sum\limits_{n=0}^{\infty}u(x-\mx,n\mt).
		\]
		Since $u(x,0)=0$ for $x>0$ we have that
		\[
			\sum\limits_{n=0}^{\infty} u(x,n\mt) = \sum\limits_{n=0}^{\infty}u(x-\mx,n\mt).
		\]
\end{enumerate}
\end{proof}

Once again, we could study the solutions' relationship to probability distributions.

\begin{thm}\label{t:dis-pmf}
Let $u(x,t)$ be a solution of \eqref{e:dis-time}. Then the space and time sections $\mx u(x,\cdot)$ and $\mt u(\cdot,t)$ form probability mass functions if and only if the assumptions (D1),
\begin{equation}\label{e:D2}
\frac{A\mx}{k}=1, \tag{D2}
\end{equation}
and
\begin{equation}\label{e:D3}
A\mx=1, \tag{D3}
\end{equation}
hold.
\end{thm}

\begin{proof}
Lemma \ref{l:dis-preservation} yields that the solutions are nonnegative and conserve sums. It suffices to include step-lengths $\mx$ and $\mt$ and identify conditions under which $\mx\sum_x u(x,0)=1$ and $\mt\sum_t u(0,t)=1$. Given the initial condition, the former sum is equal to $A\mx$. Hence the assumption (D3). Finally, since $u(-\mx,t)=0$, the equation \eqref{e:dis-time-2} implies that $u(0,n\mt)=A\prn{1-\frac{k\mt}{\mx}}^n$. Consequently, 
\[
	1=A\mt \sum\limits_{n=0}^{\infty}  \prn{1-\frac{k\mt}{\mx}}^n=\frac{A\mx}{k}.
\]
\end{proof}

\begin{figure}[t]
\vspace{42pt}
\begin{center}
\includegraphics[height=1.7in,width=1.7in]{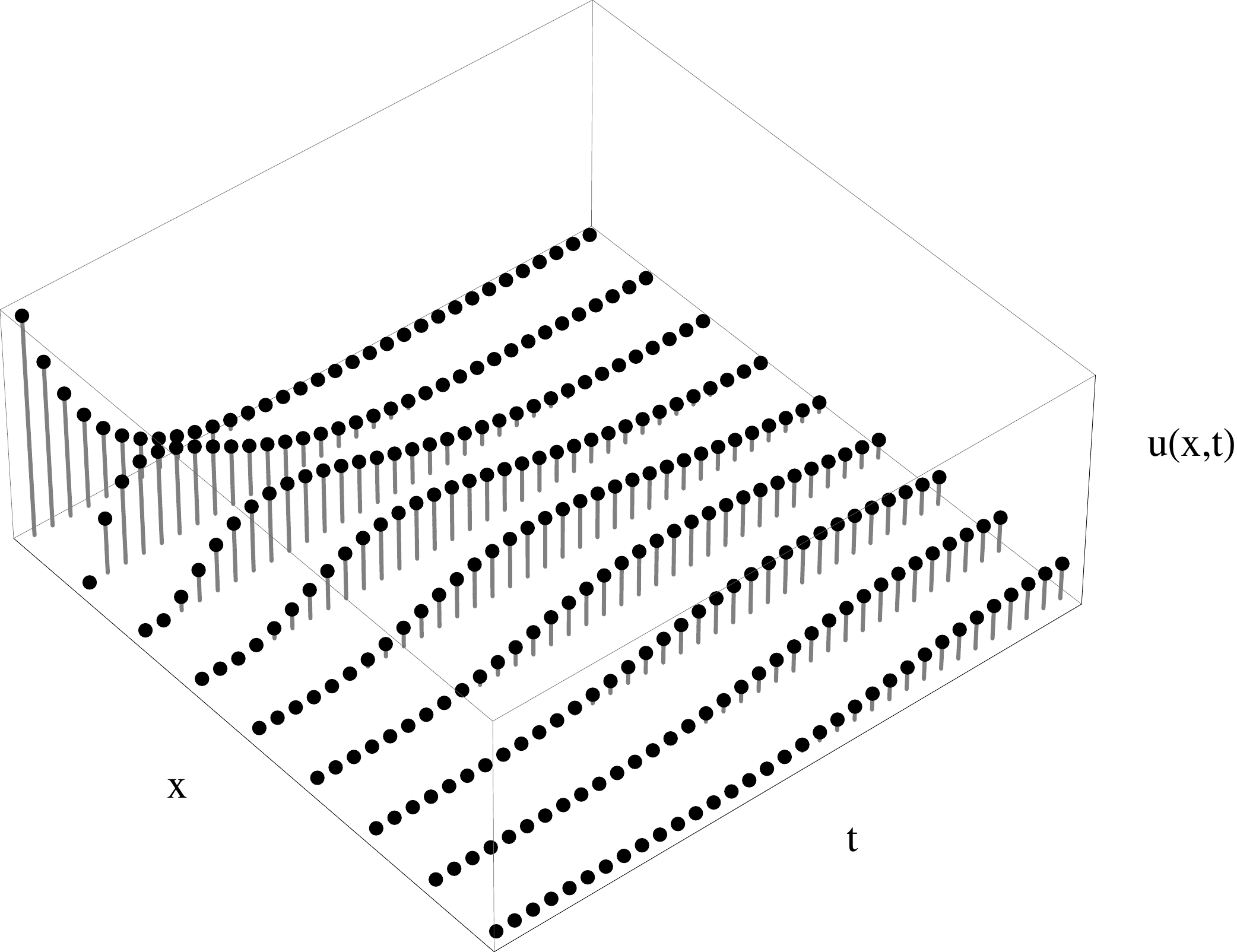} \hspace{0.2in} \ 
\includegraphics[height=1.7in,width=1.7in]{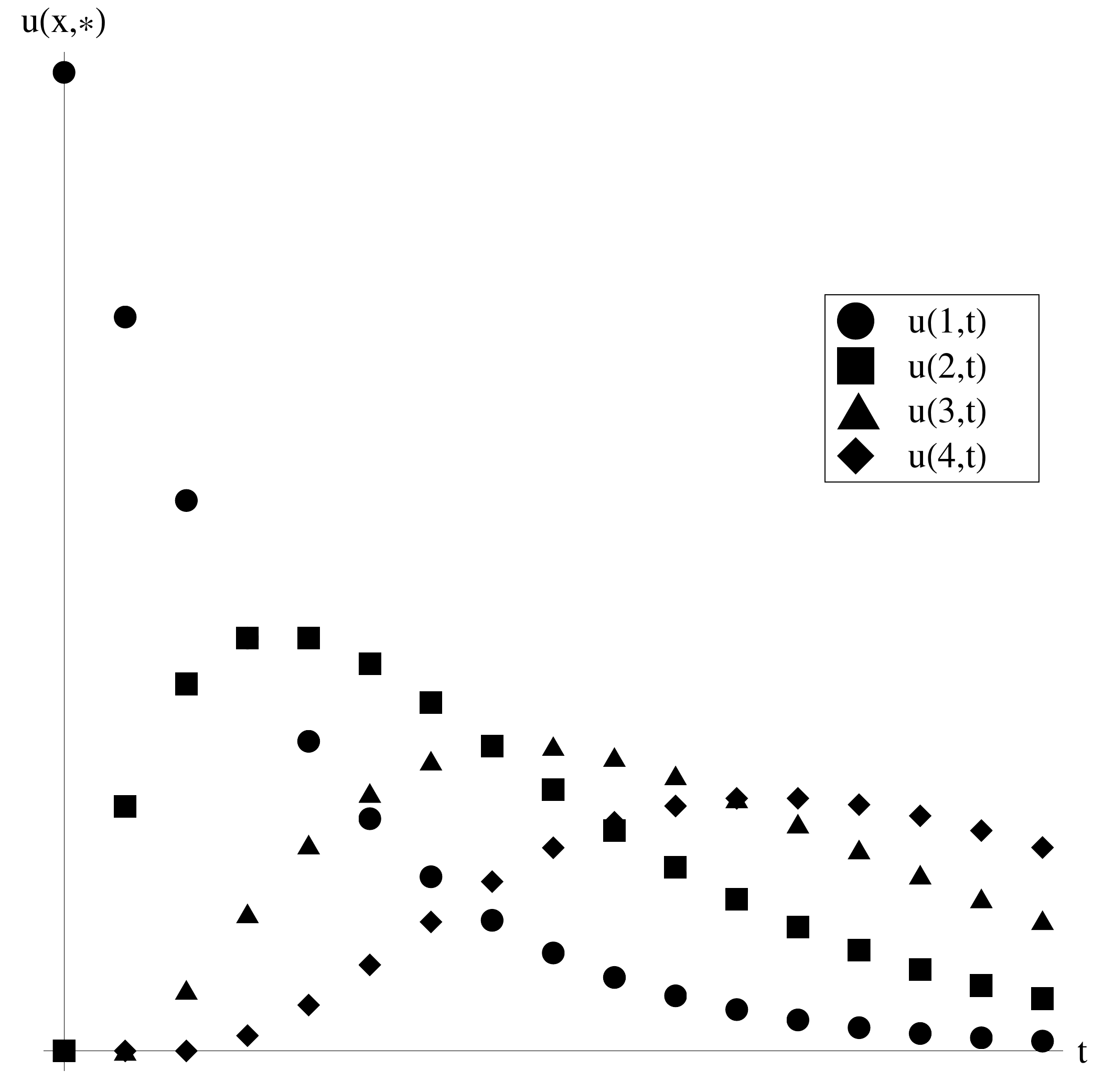} \hspace{0.2in} \
\includegraphics[height=1.7in,width=1.7in]{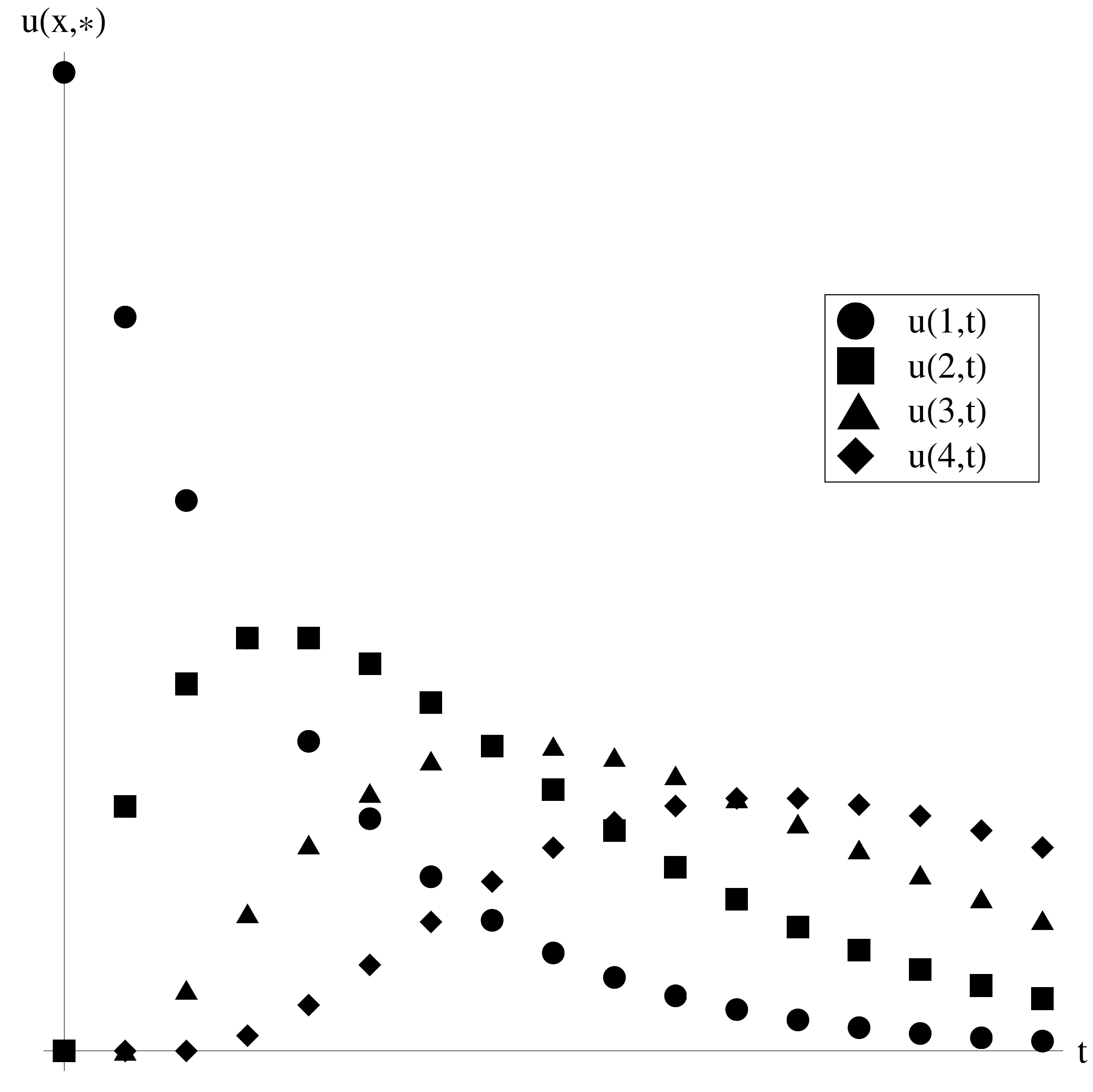}
\end{center}

\caption{Solution of the transport equation with discrete space and discrete time \eqref{e:dis-time} with $A=1$, $k=1$, $\mt=.25$ and $\mx=1$.}
\label{f:dis-time}
\end{figure}


\begin{cor}\label{c:dis-pmf}
Let $u(x,t)$ be a solution of \eqref{e:dis-time}. Then the space and time sections $\mx u(x,\cdot)$ and $\mt u(\cdot,t)$ form probability mass functions if and only if $k=1$, $\mt<\mx$ and $A=\frac{1}{\mx}$.
\end{cor}
\begin{proof}
(D2) and (D3) hold if and only if $k=1$. Consequently, (D1) could be satisfied if and only if $\mt<\mx$.
\end{proof}

Closer examination again reveals that the sections form probability mass functions of discrete probability distributions related to Bernoulli counting processes.

\begin{rem}\label{r:dis-bernoulli}
Let us consider the solution \eqref{e:dis:solution}. If we put $A=k=\mx=1$ and $\mt=p$ we get
\[
u(n,m\cdot p) = { n \choose m } \left( 1 - p \right)^{n-m}p^{m},  \quad n \geq m.
\]
which forms, for each fixed $n\in\Natn$, a probability mass function of the binomial distribution. Similarly, for each fixed $m\in\Natn$, $p=\mt$-multiple forms a probability mass function of a version of the negative binomial distribution (the value $p\cdot u(n,m\cdot p)$ describes a probability that for $m$ failures we need $n$ trials). Consequently, the solution of \eqref{e:dis-time} describes a counting Bernoulli stochastic process (see \cite{bBer}).
\end{rem}

\section{Discrete Space and General Time}\label{s:ts-time}
Let us extend the results from the last two sections by considering more general time structures. Let $\TS$ be a time scale such that $\min{\TS}=0$ and $\sup{\TS}=+\infty$. In this paragraph we consider domains
\[
\Omega=\brc{(x,t): x\in\mx\Int, t\in\TS},
\]
and the problem:

\begin{equation}\label{e:ts-time}
\begin{cases}
u^\ddt(x,t)+k \nabla_x u(x,t) = 0, \quad (x,t)\in\Omega,\\
u(x,0)=\begin{cases}
A, \quad x=0,\\
0, \quad x\neq 0,
\end{cases}
\end{cases}
\end{equation}
where $A>0$, $k>0$ and $\nabla_x u(x,t)$ is the backward difference defined in \eqref{e:nabla} and $u^\ddt$ is the delta-derivative in time variable. Since the space is discrete, we could again rewrite the equation in \ref{e:ts-time} into
\begin{equation}\label{e:ts-time-2}
u^\ddt(x,t)=-\frac{k}{\mx}\prn{u(x,t)-u(x-\mx,t)}.
\end{equation}
In order to conserve the sign of solutions we assume that
\begin{equation}\label{e:TS1}
1-\frac{k\mt(t)}{\mx}>0, \tag{TS1}
\end{equation}
i.e. the condition which is similar to the positive regressivity in the time scale theory (e.g. \cite[Section 2.2]{bBP}) or the so-called CFL condition in the discretization of the transport equation (e.g. \cite[Section 4.4]{bLev}).

Let $u$ be a solution of \eqref{e:ts-time}. One could use \cite[Proposition 5.2]{aMB} to show that $u(x,t)=0$ for all $x<0$ is the unique solution there. Since $u(-\mx,t)=0$, we could see that $u^\ddt(0,t)=-\frac{k}{\mx} u(0,t)$. Given the initial condition and assumption \eqref{e:TS1}, we get 
$u(0,t)=Ae_{-\frac{k}{\mx}}(t;0)$, where $e_{-\frac{k}{\mx}}(t;0)$ is a time scale exponential function (see \cite[Section 2]{bBP}).

\begin{lem}\label{l:ts-first-branch}
The solution of \eqref{e:ts-time} satisfies
\begin{enumerate}[\upshape (i)]
	\item $\lim\limits_{t\To\infty} u(0,t)=0$,
	\item $\int\limits_0^\infty u(0,t) \Delta t=A\frac{\mx}{k}$.
\end{enumerate}
\end{lem}

\begin{proof}
\begin{enumerate}[\upshape (i)]
	\item Follows directly from the assumption \eqref{e:TS1} and the properties of the exponential functions \cite[Section 2.2]{bBP},
	\item \[
		\int\limits_0^\infty u(0,t) \Delta t=A \int\limits_0^\infty e_{-\frac{k}{\mx}}(t;0) \Delta t = \lim\limits_{t\To\infty} -A\frac{\mx}{k} \prn{e_{-\frac{k}{\mx}}(t;0)-1} = A\frac{\mx}{k}.
	\]
\end{enumerate}
\end{proof}

Unique solutions of $u(m\mx,t)$ could be found using the variation of constants (see e.g. \cite[Theorem 2.77]{bBP}). However, these computations depend critically on a particular time scale and cannot be performed in general. For example, one could compute that the second branch of the solution has the form
\[
u(\mx,t)=A\frac{k}{\mx} e_{-\frac{k}{\mx}}(t;0) \int_0^t \frac{\Delta \tau}{1-\frac{k\mt(\tau)}{\mx}}.
\]
This implies that we can't derive closed-form solutions as in previous sections. Formally, these solutions can be expressed as Taylor-like series with generalized polynomials whose form depends on particular time scale (see \cite{aMB} and \cite[Section 1.6]{bBP}). We determine these solutions in special cases (see Lemmata \ref{l:con-solution}, \ref{l:dis-solution} and \ref{l:het-discrete}). Therefore, we are forced use another means to show the properties of solutions we are interested in.

\begin{lem}\label{l:ts-sign-1}
Let $x\in \mx \Nat$. If \eqref{e:TS1} is satisfied and $u(x-\mx,t)\geq 0$ and for all $t\in\TS$  and $u(x-\mx,t)> 0$ at least for one $t\in\TS$ then $u(x,t)\geq0$ for all $t\in\TS$.
\end{lem}

\begin{proof}
First, note that $u(x,0)=0$ for all $x>0$. Consequently, \eqref{e:ts-time-2} implies that $u^\ddt(x,t)>0$ at the beginning of the support of $u(x-\mx,t)$ and $u(x,t)$ is strictly increasing there.
\begin{itemize}
	\item If $t$ is right-scattered then we can rewrite the equation \eqref{e:ts-time-2} into
		\[
			u(x,t+\mt)= \prn{1-\frac{k\mt}{\mx}} u(x,t) + \frac{k\mt}{\mx} u(x-\mx,t).
		\]
		If $u(x,t)\geq 0$, then this is the weighted average of two nonnegative values and thus nonnegative as well.
	\item If $t$ is right-dense then the equation \eqref{e:ts-time-2} has the form
		\[
			u_t(x,t)=-\frac{k}{\mx} u(x,t)+ \frac{k}{\mx} u(x-\mx,t).
		\]
		Since both $u(x-\mx,t)\geq 0$ and $u(x,t)\geq 0$, we have that $u_t(x,t)\geq -\frac{k}{\mx} u(x,t)$ and thus $u(x,t)$ cannot become negative.
\end{itemize}
Following the induction principle (e.g. \cite[Theorem 1.7]{bBP}), we could see that $u(x,t)\geq 0$ for all $t\in\TS$.
\end{proof}

Lemma \ref{l:ts-sign-1} serves as the inductive step in the proof of the sign-conservation.

\begin{thm}\label{t:ts-sign-2}
If \eqref{e:TS1} holds then $u(x,t)\geq0$ for all $(x,t)\in\Omega$.
\end{thm}

\begin{proof}
We prove the statement by mathematical induction. Firstly, $u(0,t)=Ae_{-\frac{k}{\mx}}(t;0)>0$. Secondly, if $u(x,t)\geq 0$ then Lemma \ref{l:ts-sign-1} implies that $u(x+\mx,t)\geq 0$ which finishes the proof.
\end{proof}

The following auxiliary lemma shows that the variation of constant formula which generates further branches of solutions conserve zero-limits at infinity.

\begin{lem}\label{l:ts-integral-1}
Let us consider a time scale $\TS$, a constant $K$ such that $1-\mu K>0$ and a function $f:\TS\To[0,\infty)$ such that the integral $\int\limits_0^\infty f(t) \Delta t$ is finite. If we define $g:\TS\To[0,\infty)$ by
\[
	g(t)=\int\limits_0^t e_{-K}(t,\sigma(\tau)) f(\tau)\Delta\tau,
\]
then $\lim\limits_{t\To\infty} g(t)= 0$.
\end{lem}

\begin{proof}
Since $\int\limits_0^\infty f(t) \Delta(t)$ is finite we know that for each $\epsilon>0$ there exists $T>0$ such that for all $t\in\TS$, $t>T$ the inequality
\begin{equation}\label{e:ts-integral-1}
\int\limits_t^\infty f(\tau) \Delta\tau < \frac{\epsilon}{2},
\end{equation}
holds. Similarly, properties of time scale exponential function imply that for each $\epsilon>0$ and $T>0$ there exists $R>T$ such that for all $t\in\TS$, $t>R$ the following inequality is satisfied
\begin{equation}\label{e:ts-integral-2}
\int\limits_0^T e_{-K}(t;\sigma(\tau)) \Delta\tau < \frac{\epsilon}{2F},
\end{equation}
with $F=\max\limits_{t\in\TS} f(t)$. Consequently, inequalities \eqref{e:ts-integral-1} and \eqref{e:ts-integral-2} imply that for each for each $\epsilon>0$ there exists $T>0$ and $R>T$ such that for all $t>R$
\begin{align*}
g(t) &= \int\limits_0^t e_{-K}(t,\sigma(\tau)) f(\tau)\Delta\tau \\
	&= \int\limits_0^T e_{-K}(t;\sigma(\tau)) f(\tau) \Delta\tau + \int\limits_T^t e_{-K}(t;\sigma(\tau)) f(\tau) \Delta\tau\\
	&\leq F \int\limits_0^T e_{-K}(t;\sigma(\tau)) \Delta\tau + \int\limits_T^t f(\tau) \Delta\tau\\
	&< F \frac{\epsilon}{2F} + \frac{\epsilon}{2}\\
	&=\epsilon,
\end{align*}
which implies that $\lim\limits_{t\To\infty} g(t)= 0$.
\end{proof}

Consequently, we are able to show that the integrals are constant for each fixed $x\geq 0$.

\begin{thm}\label{t:ts-integral-2}
If \eqref{e:TS1} holds and $u(x,t)$ is a solution of \eqref{e:ts-time} then
\[
\int\limits_0^\infty u(x,t) \Delta t = \int\limits_0^\infty u(0,t) \Delta t=A\frac{\mx}{k},
\]
for all $x\in\mx\Natn$.
\end{thm}

\begin{proof}
We proceed by mathematical induction.
\begin{itemize}
	\item For $x=0$ the convergence of the integral to $A\frac{\mx}{k}$ follows from Lemma \ref{l:ts-first-branch} (ii).
	\item Let us fix $x\in\mx\Nat$ and assume that the statement holds for a function $u(x-\mx,t)$. If we integrate \eqref{e:ts-time-2} we get
	\begin{equation}\label{e:ts-integral-3}
		\int_0^\infty u^\ddt(x,\tau) \Delta\tau=-\frac{k}{\mx}\prn{\int_0^\infty u(x,\tau)\Delta\tau-\int_0^\infty u(x-\mx,\tau)\Delta\tau}.
	\end{equation}
	Let us concentrate on the left-hand side term. The variation of constants formula (\cite[Theorem 2.77]{bBP}) implies that
	\[
		u(x,t) = \int\limits_0^t e_{-\frac{k}{\mx}}(t,\sigma(\tau)) u(x-\mx,\tau)\Delta\tau.
	\]
	Consequently, Lemma \ref{l:ts-integral-1} implies that $ \lim\limits_{t\To\infty} u(x,t)=0$. Using the initial condition $u(x,0)=0$, we could rewrite the left-hand side of \eqref{e:ts-integral-3} into
	\[
		\int_0^\infty u^\ddt(x,\tau) \Delta\tau = \lim\limits_{t\To\infty} u(x,t) -u(x,0) = 0.
	\]
	This implies that \eqref{e:ts-integral-3} could be rewritten into
	\[
		0 =-\frac{k}{\mx}\prn{\int_0^\infty u(x,\tau)\Delta\tau-\int_0^\infty u(x-\mx,\tau)\Delta\tau},
	\]
	or equivalently into
	\[
	\int_0^\infty u(x,\tau)\Delta\tau=\int_0^\infty u(x-\mx,\tau)\Delta\tau,
	\]
	which finishes the proof.	
\end{itemize}
\end{proof}

\begin{figure}[t]
\vspace{42pt}
\begin{center}
\includegraphics[height=1.7in,width=1.7in]{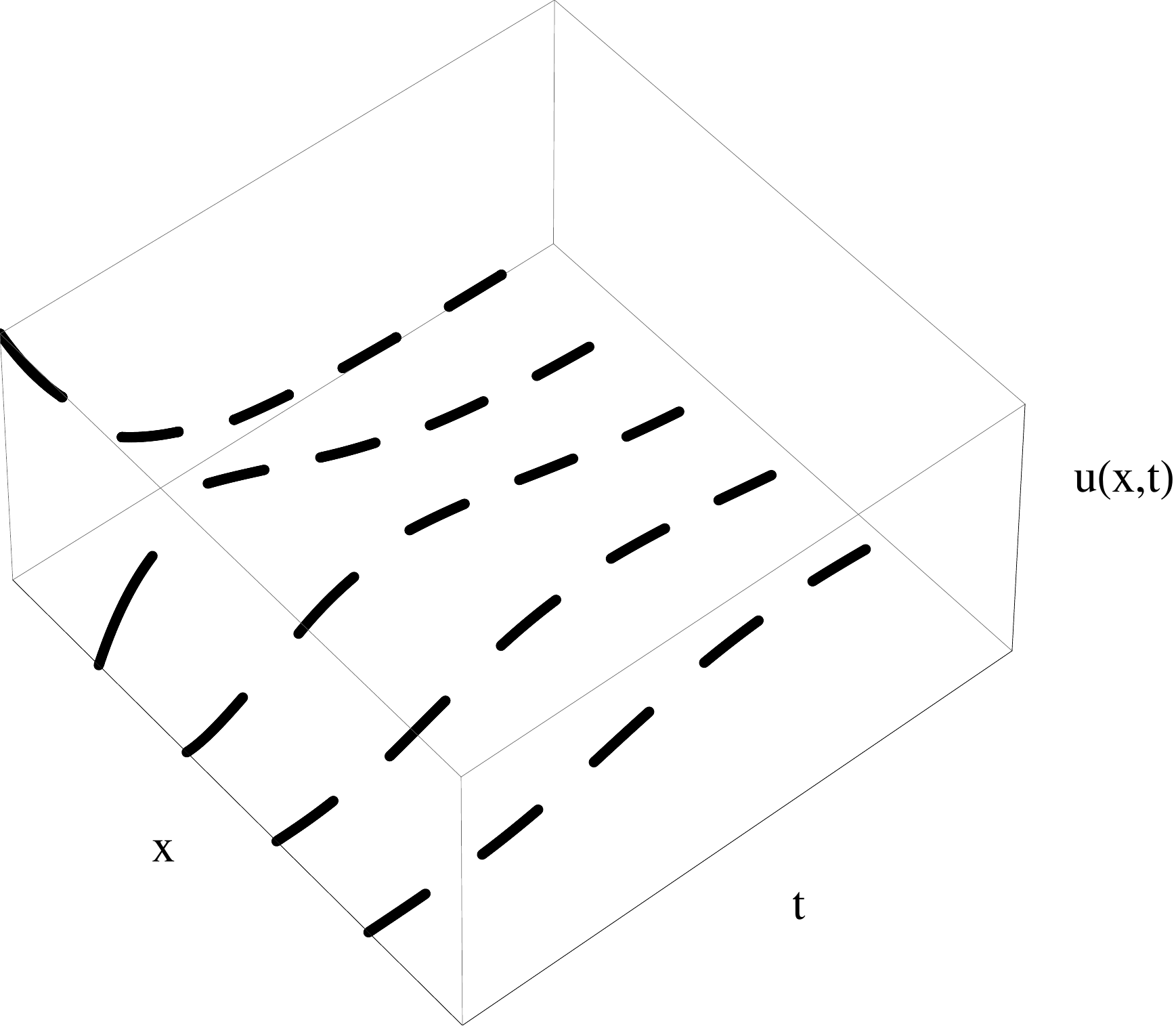} \hspace{0.2in} \ 
\includegraphics[height=1.7in,width=1.7in]{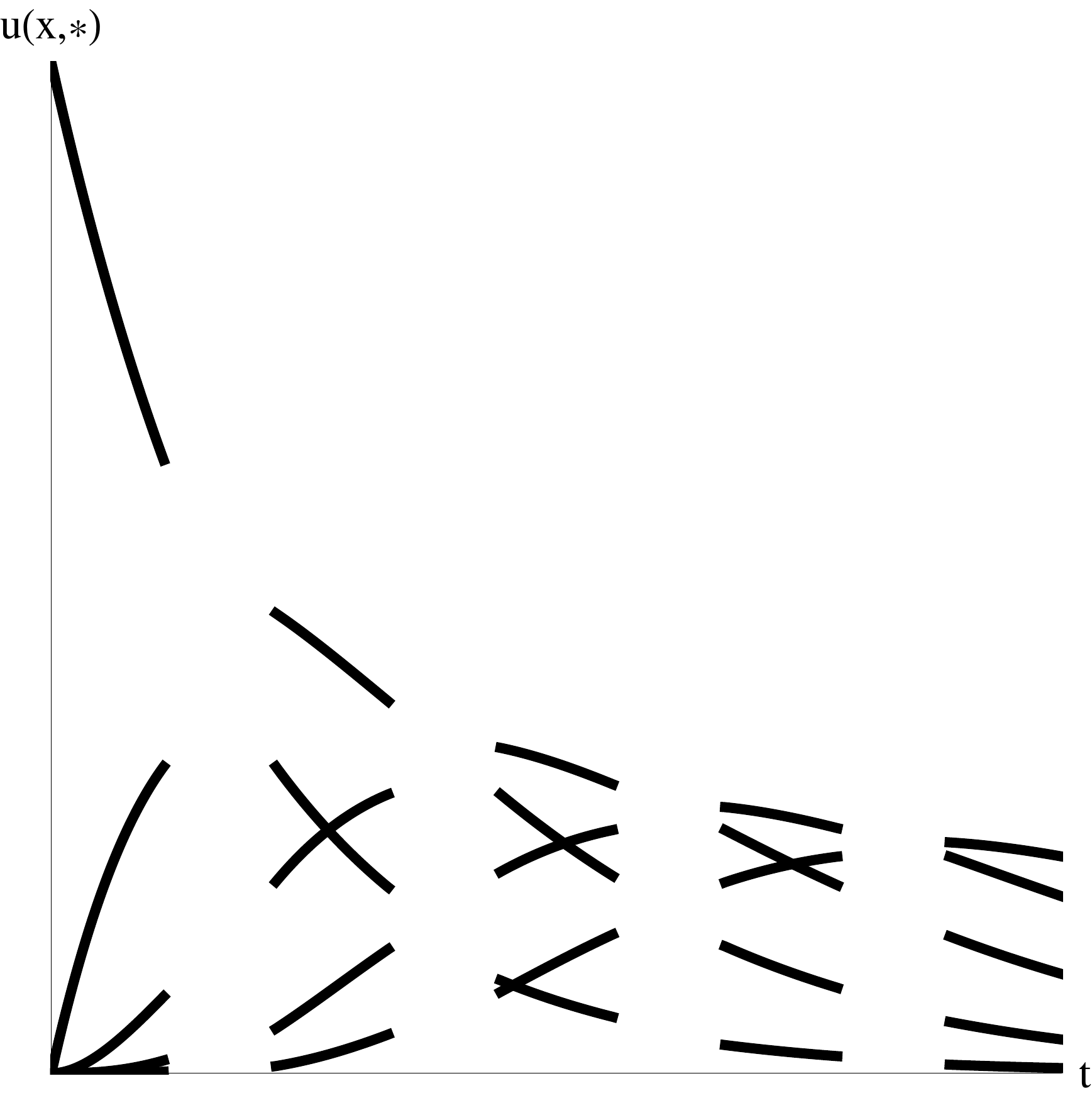} \hspace{0.2in} \
\includegraphics[height=1.7in,width=1.7in]{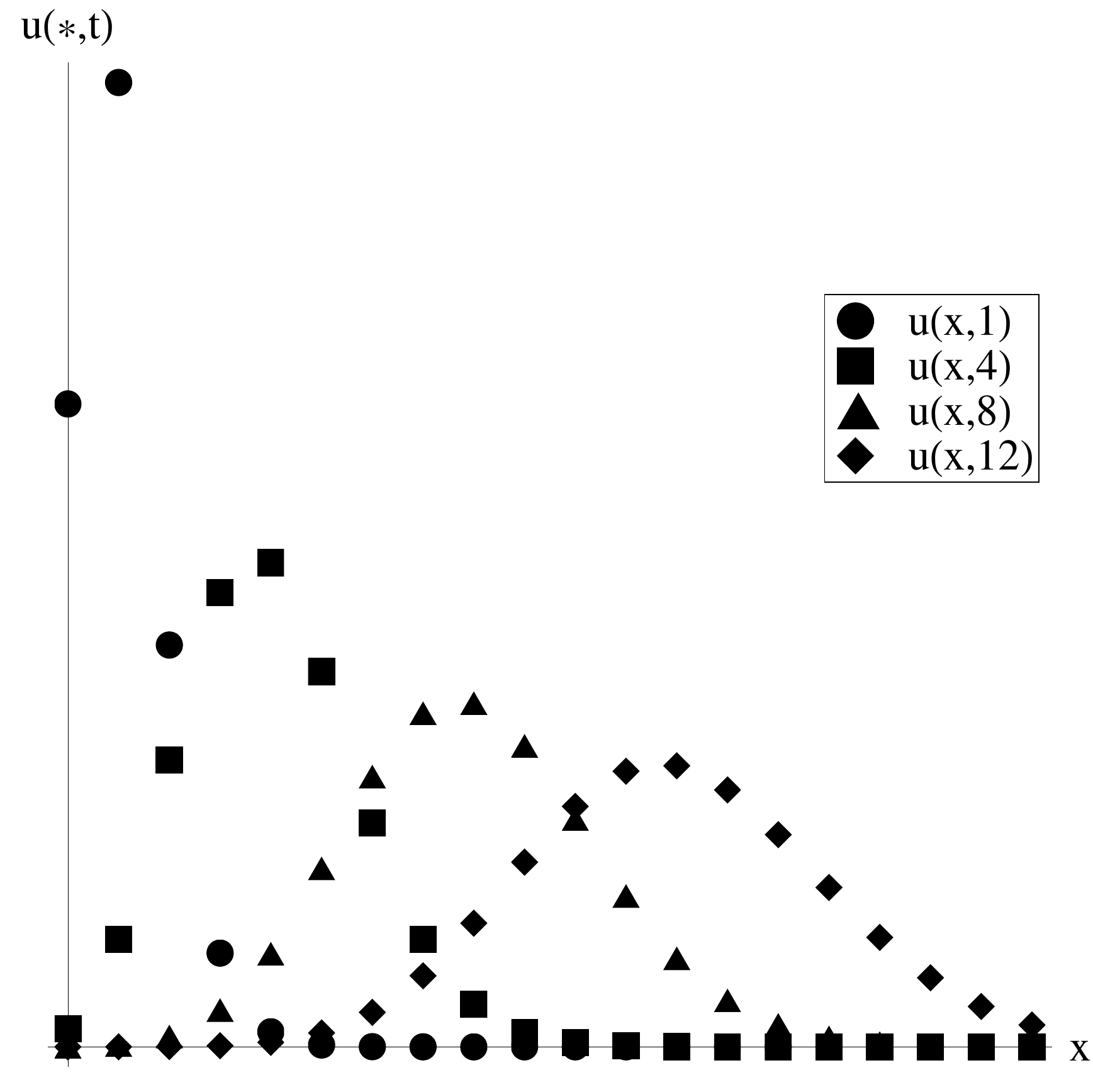}
\end{center}

\caption{Solution of the transport equation with discrete space and general time  \eqref{e:ts-time} with $A=1$, $\mx=1$, $k=1$ and $\TS=\bigcup\limits_{i=0}^{\infty} \bra{i,i+\frac{1}{2}}$.}
\label{f:ts-time}

\end{figure} 


Finally, we show that the integrals (sums in this case) remains constant in time as well.

\begin{thm}\label{t:ts-integral-3}
If \eqref{e:TS1} holds and $u(x,t)$ is a solution of \eqref{e:ts-time} then
\[
\int\limits_0^\infty u(x,t) \Delta x = \mx\sum\limits_{m=0}^{\infty}u(m\mx,t)=A\mx,
\]
for all $t\in\TS$.
\end{thm}
\begin{proof}
Let $u(x,t)$ be a solution of \eqref{e:ts-time}. We define a function $S:\TS\To\Real$ by
\[
S(t):=\int\limits_0^\infty u(x,t) \Delta x = \mx\sum\limits_{m=0}^{\infty}u(m\mx,t),
\]
and show that $S^\ddt(t)=0$ for all $t\in\TS$.

We can rewrite the equation in \eqref{e:ts-time} into
\[
u^\ddt(x,t)=-\frac{k}{\mx} u(x,t)+\frac{k}{\mx} u(x-\mx,t).
\]
Consequently,
\begin{align}
S^\ddt(t)&= \mx\sum\limits_{m=0}^{\infty}u^\ddt(m\mx,t) \label{e-ts-s1}\\
	&=-k \sum\limits_{m=0}^{\infty}u(m\mx,t) + k \sum\limits_{m=0}^{\infty}u((m-1)\mx,t) \label{e-ts-s2} \\
	&= 0. \label{e-ts-s3}
\end{align}
We have to justify the first equality \eqref{e-ts-s1}, i.e. the interchangability of the delta-derivative and summation at each $t_0\in\TS$. If $t_0$ is right-scattered the non-negativity of the solution implies
\begin{align*}
S^\ddt(t_0)&=\frac{\mx\sum\limits_{m=0}^{\infty}u(m\mx,t_0+\mt(t_0))-\mx\sum\limits_{m=0}^{\infty}u(m\mx,t_0)}{\mt(t_0)}\\
&=\mx\sum\limits_{m=0}^{\infty} \frac{u(m\mx,t_0+\mt(t_0))-u(m\mx,t_0)}{\mt(t_0)}  = \mx\sum\limits_{m=0}^{\infty}u^\ddt(m\mx,t_0).
\end{align*} 
If $t_0$ is right-dense and there is a continuous interval $[t_0,s]$, $s>t_0$, we show that the sum $\sum\limits_{m=0}^{\infty}u^\ddt(m\mx,t)$ converge uniformly on $[t_0,s]$. First, let us note that \eqref{e-ts-s2} yields that this is implied by the uniform convergence of $\sum\limits_{m=0}^{\infty}u(m\mx,t)$. One could use Corollary \ref{c:con-general-ic} to get ($\kappa=k/\mx$):
\begin{align*}
\sum\limits_{m=0}^{\infty}u(m\mx,t) &= \sum\limits_{m=0}^{\infty} \prn{e^{-\kappa (t-t_0)} \sum\limits_{i=0}^{m} C_{i} \frac{(\kappa (t-t_0))^{m-i}}{(m-i)!}} \\
&= e^{-\kappa (t-t_0)}  \sum\limits_{m=0}^{\infty} \frac{(\kappa (t-t_0))^{m}}{m!} \cdot \sum\limits_{i=0}^{\infty}  C_i.
\end{align*}
If $\sum_{i=0}^{\infty}  C_i$ is finite (i.e. $S(t_0)$ is finite), then this sum converge uniformly on an arbitrary closed interval. Finally, if $t_0$ is right-dense and there is no continuous interval $[t_0,s]$, $s>t_0$, we consider a function $v(x,t)$ with $v(m\mx,t_0)=u(m\mx,t_0)$ for all $m$ such that $v$ is a solution on a domain with a continuous interval $[t_0,s]$, $s>t_0$. Obviously, the equation in \eqref{e:ts-time} implies that $v_t(m\mx,t_0)=u^\ddt(m\mx,t_0)$ for all $m$. Moreover for each $\delta>0$ there is $\theta>0$ such that for all $t\in[t_0,t_0+\theta]_\TS$:
\[
(1-\delta) \sum\limits_{m=0}^{\infty}v(m\mx,t) \leq \sum\limits_{m=0}^{\infty}u(m\mx,t) \leq (1+\delta) \sum\limits_{m=0}^{\infty}v(m\mx,t).
\]
Consequently,
\begin{align*}
0 &= \sum\limits_{m=0}^{\infty}u^\ddt(m\mx,t_0) =\sum\limits_{m=0}^{\infty}v_t(m\mx,t_0) \\
&= \prn{\sum\limits_{m=0}^{\infty}v(m\mx,t_0)}_t=\prn{\sum\limits_{m=0}^{\infty}u(m\mx,t_0)}^\ddt.
\end{align*}
Taking into account the fact that $u(x,0)$ is given by the initial condition in \eqref{e:ts-time}, we see that $S(0)=A\mx$. Consequently, \eqref{e-ts-s1}-\eqref{e-ts-s3} imply that $S(t)=A\mx$.
\end{proof}
We could now study the relationship with probability distributions and we begin by generalizing probability density and mass functions.
\begin{def}\label{d:pgf}
We say that a function $f:\TS\To\Real_0^+$ is a \emph{dynamic probability density function} if
\[
	\int_{-\infty}^{\infty} f(t) \Delta t = 1.
\]
\end{def}
Note that if $\TS=\Real$ then $f$ is a probability density function. If $\TS=\mt\Int$ then $\mt f$ is a probability mass functions (see Therorem \ref{t:dis-pmf}).

Combining Theorems \ref{t:ts-integral-2} and \ref{t:ts-integral-3} we get the necessary and sufficient condition for sections to generate probability distributions.
\begin{lem}\label{l:ts-pdf}
Let $u(x,t)$ be a solution of \eqref{e:ts-time}.
\begin{enumerate}
\item $u(\cdot,t)$ is a dynamic probability density function for all $x\in\mx\Natn$ if and only if $\frac{A\mx}{k}=1$ and $\mt(t)<\mx$ for all $t\in\TS$.
\item $u(x,\cdot)$ is a dynamic probability density function for all $t\in\TS$ if and only if $A\mx=1$ and \eqref{e:TS1} holds.
\end{enumerate}
\end{lem}
\begin{proof}
The proof is a direct application of Theorems \ref{t:ts-integral-2} and \ref{t:ts-integral-3}.
\end{proof}

Finally, we provide the necessary and sufficient condition for both sections.
\begin{thm}\label{t:ts-pdf}
Let $u(x,t)$ be a solution of \eqref{e:ts-time}. Then both $u(x,\cdot)$ and $u(\cdot,t)$ are dynamic probability density functions for all  $t\in\TS$ and $x\in\mx\Natn$  if and only if $k=1$, $A\mx=1$ and $\mt(t)<\mx$ for each $t\in\TS$.
\end{thm}
\begin{proof}
The proof follows from Lemma \ref{l:ts-pdf}.
\end{proof}

\section{Applications}\label{s:processes}
As suggested in Remarks \ref{r:con-poisson} and \ref{r:dis-bernoulli} the time and space sections of solutions of the transport equation on various domains generate important probability distributions (cf. Table \ref{t:sections}).

\begin{center}
\begin{table}
\begin{tabular}{|c||c|c|c|} \hline
& $u(\cdot,t)$& $u(0,\cdot)$& $u(x,\cdot) , x\geq 0$ \\ \hline\hline
$\Int\times\Real$ & Poisson dist. & exponential dist. & Erlang (Gamma) dist. \\ \hline
$\Int\times p\Int$ & binomial dist. & geometric dist. & negative binomial dist. \\ \hline
\end{tabular}
\\[2mm]
\caption{Correspondence of time and space sections with probability distributions}\label{t:sections}
\end{table}
\end{center}

In other words, the solutions correspond to the so-called counting stochastic processes describing number of occurrences of certain random events (arrival of customers in a queue, device failures, phone calls, scored goals, etc.) (e.g. \cite[Chapters 4 and 5]{bPan}, \cite{bGha}). They have following properties
\begin{enumerate}
\item probability of number of events (occurrences) at time $t$ is given by $u(\cdot,t)$ (Poisson distribution, binomial distribution),
\item probability distribution of the time of the first occurrence is given by  $u(0,t)$ (exponential or geometric distribution),
\item probability distributions that at least $x$ events have happened until time $t$ are given by $u(x-1,\cdot)$ (Erlang or negative binomial distributions),
\item probability distribution of the waiting time until the next occurrence is given by  $u(0,t)$ (exponential or geometric distribution).
\end{enumerate}

Our analysis in Section \ref{s:ts-time}, summarized in Theorem \ref{t:ts-pdf}, suggests that properties (1)-(3) are conserved on general domains $\Int\times\TS$. Properties (2)-(3) are conserved in the sense of Definition \ref{d:pgf} (see Examples \ref{x:het-bernoulli} and \ref{x:stop-start} below). Property (4) does not apply because of the underlying inhomogeneous time structure.



The convergence relationship between the distributions from Table \ref{t:sections} is well-known \cite{bGha}. Our analysis strengthens this relationship since the convergence is based on the solution of the same partial equation with changing underlying structures.

We conclude this section by suggesting two applications which emphasize the time scale choice. First, let us consider Bernoulli trials with non-constant probability of successes. For example, \cite{aDR} shows that the probability that a goal is scored in each minute of the association football match is not constant but increases throughout the game, especially in the last minutes of each half-time. Let us derive an explicit solution on arbitrary heterogeneous discrete structure.
\begin{lem}\label{l:het-discrete}
Let us consider a heterogenous discrete time scale $\TS=\brc{0,\mu_1,\mu_1+\mu_2,\ldots, \sum_{i=1}^{n}\mu_i ,\ldots}$. Then the solution of \eqref{e:ts-time} has the form
\begin{equation}\label{e:het-discrete}
u\prn{m\mx,\sum\limits_{i=1}^{n}\mu_i} = A \sum\limits_{\pi\in P^{n-m}_m} \prod\limits_{i=1}^n K_i^{\pi_i} L_i^{1-\pi_i},
\end{equation}
where $K_i=1-k\frac{\mu_i}{\mu_x}$, $L_i=k\frac{\mu_i}{\mu_x}$ and $P^q_r$ denote a set of all permutation vectors containing $q$ ones and $r$ zeros.
\end{lem}
\begin{proof}
We base our proof on the relationship
\[
u\prn{m\mx,\sum\limits_{i=1}^{n}\mu_i} =\prn{1-\frac{k\mu_n}{\mx}}u(x,t)+\frac{k\mu_n}{\mx}u(x-\mx,t)
\]
and proceed by induction. First, the initial condition implies that the statement holds for $n=0$. Next, let us assume that the statement holds for $n\in\Natn$, i.e. \eqref{e:het-discrete} is satisfied. Then we have $u\prn{m\mx,\sum_{i=1}^{n+1} \mu_n}=0$ for $m\notin \brc{0,1,\ldots,n+1}$. Furthermore, for $m=0$
\[
	u\prn{0,\sum\limits_{i=1}^{n+1}\mu_i} = K_{n+1} \prn{A K_1 K_2\ldots K_n}+0 = A K_1 K_2\ldots K_n K_{n+1}.
\]
Next, for $m\in\prn{1,2,\ldots, n}$:
\begin{align*}
u\prn{m\mx,\sum\limits_{i=1}^{n+1}\mu_i} &= K_{n+1} A \sum\limits_{\pi\in P^{n-m}_m} \prod\limits_{i=1}^n K_i^{\pi_i} L_i^{1-\pi_i} + L_{n+1} A \sum\limits_{\pi\in P^{n-m+1}_{m-1}} \prod\limits_{i=1}^n K_i^{\pi_i} L_i^{1-\pi_i}\\
&= A \sum\limits_{\pi\in P^{n+1-m}_m} \prod\limits_{i=1}^{n+1} K_i^{\pi_i} L_i^{1-\pi_i}.
\end{align*}
Finally, for $m=n+1$ we have
\[
	u\prn{(n+1)\mx,\sum\limits_{i=1}^{n+1}\mu_i} = 0+L_{n+1} \prn{A L_1 L_2\ldots L_n} = A L_1 L_2\ldots L_n L_{n+1}.
\]
\end{proof}

We could immediately apply this result to obtain generalizations of standard Bernoulli processes.

\begin{exmp}\label{x:het-bernoulli}
\textbf{Heterogeneous Bernoulli process.}
Let us consider a repeated sequence of trials and assume that the probability of success $p_i$ in $i$-th trial is non-constant, in contrast to standard Bernoulli process discussed in Section \ref{s:dis-time}. If we construct a discrete time scale 
\[
\TS= \brc{0,p_1,p_1+p_2,\ldots,\sum\limits_{i=1}^{n}p_i,\ldots},
\]
then the solution $u(x,t)$ of \eqref{e:ts-time} generates the probability distributions discussed above. Let us choose, for example, $A=\mx=k=1$. Then, $u(\cdot,\sum_{i=1}^{n-1}p_i)$ is the probability mass function describing number of successes in the first $n$ trials. Moreover, $p_i u(x,\cdot)$ is the probability mass function of the number of trials needed to get $x+1$ successes.

To illustrate, let us choose $\TS=\brc{0,\frac{1}{2},\frac{1}{2}+\frac{1}{3},\ldots,\sum\limits_{i=1}^{n}\frac{1}{i+1},\ldots}$ 
to study a process in which the probability of successful trial decreases harmonically. We could use Lemma \ref{l:het-discrete} to determine that:
\[
u\prn{m,\sum\limits_{i=1}^{n}p_i} = \sum\limits_{\pi\in P^{n-m}_m} \prod\limits_{i=1}^n \prn{\frac{1}{i+1}}^{\pi_i} \prn{\frac{i}{i+1}}^{1-\pi_i}, \quad 0\leq m\leq n.
\] 
For example, the probability mass function for the first successful trial appearing in $k$-th trial, i.e. $p_i u(0,\cdot)$, has the form
\[
f(k)=\frac{1}{k(k+1)},\quad k\in\Nat.
\]
\end{exmp}

Finally, we consider a mixed time scale, which, coupled with the transport equation, generates mixed processes and distributions.
\begin{exmp}\label{x:stop-start}
\textbf{Stop-Start Bernoulli-Poisson Process.} Let us assume that a device is regularly used throughout a constant period and then switched off for another one. Let us assume that the probability of failure when the device is in use is determined by a continuous process, whereas the probability of failure in the rest mode is given by a discrete process. This leads to mixed probability distributions which could be generated e.g. by
\[
\TS=\bigcup\limits_{i=0}^{\infty} \bra{i,i+\frac{1}{2}}.
\]
Again $u(x,\cdot)$ describes the mixed probability distribution of $x+1$ failures, in the sense of Definition \ref{d:pgf}. Similarly, $u(\cdot, t)$ is the probability mass function describing the number of failures at time $t$. Note that the probability of failure in the rest mode is given by the length of the discrete gap (cf. Definition \ref{d:pgf}). As in the previous example, we are not able to find the closed-form solutions but one could tediously solve the separate equations to get that:
{\allowdisplaybreaks
\begin{align*}
u(0,t) &=\frac{1}{2^n} e^{\frac{n}{2}-t}, \\
u(1,t) &=\frac{2t+n}{2^{n+1}} e^{\frac{n}{2}-t}, \\
u(2,t) &=\frac{4t^2+4nt+(n^2-4n)}{2!\cdot 2^{n+2}} e^{\frac{n}{2}-t}, \\
u(3,t) &=\frac{8t^3+12nt^2+6(n^2-4n)t+(n^3-12n^2+16n)}{3!\cdot 2^{n+3}} e^{\frac{n}{2}-t}, \\
\ldots \\
u(x,t) &= \frac{\mathrm{polynomial\ of\ order\ } x}{x!\cdot 2^{n+x}} e^{\frac{n}{2}-t}, \\
\end{align*}
}
for $n\in\Nat_0$ ($n$-th continuous part) and $t\in\bra{n,n+\frac{1}{2}}$. See Figure \ref{f:ts-time} for illustration.
\end{exmp}

\section{Conclusion and Future Directions}
There is a number of open questions related to the anaysis presented in this paper. In Section \ref{s:ts-time} we were unable to provide a general closed-form solution of problem \eqref{e:ts-time}. With the connection to probability distributions, is it possible to provide one for further special choices of $\TS$ (see e.g. Examples \ref{x:het-bernoulli} and \ref{x:stop-start})?

In the classical case, the solution is propagated along characteristics. Obviously, our analysis in Sections \ref{s:con-time} and \ref{s:dis-time} implies that this is not the case on semidiscrete domains. However, one could show that at least the maxima are propagated along characteristics on discrete-continuous or discrete-discrete domains (computing directly or using modes of probability distributions). Having no closed-form solutions on time scales, could we prove this property for an arbitrary time scale? This question is closely related to modes of the corresponding probability distributions and the question could be therefore formulated in more general way. Can we, at least in special cases, determine the descriptive statistics related to the generated probability distributions?

From the theoretical point of view, there is also a natural extension to consider a transport equation with continuous space and general time, or general space and time. The applicability of this settings is limited by the fact that such problems does not conserve sign in general (cf. assumption $\mt(t)<\mx$ in Theorem \ref{t:ts-pdf}).

\subsection*{Acknowledgements}
The authors gratefully acknowledge the support by the Czech Science Foundation, grant No. 201121757 and by the Ministry of Education, Youth and Sports of the Czech Republic, Research Plan No. MSM 4977751301.


\ \\

\noindent \textsc{University of West Bohemia, Univerzitni 22, 31200 Pilsen, Czech Republic}\\[5mm]
pstehlik@kma.zcu.cz


\begin{thebibliography}{9}
\bibitem{aAM}
    {\rm Ahlbrandt C., Morian C.,}
    {\rm Partial differential equations on time scales},
    {\em J. Comput. Appl. Math.}
    {{\bf 141}(2002), 35--55.}

\bibitem{bBer} Bertsekas, D. P. and Tsitsiklis, J. N.,
 {\em Introduction to Probability},
Athena Scientific, Massachusetts, 2002.
    
   
\bibitem{aBG}
	{\rm Bohner M. and Guseinov G.,}
    {\rm Partial differentiation on time scales,}
    {\em Dyn. Systems Appl.}
    {{\bf 13}(2004), 351--379.}

\bibitem{bBP} Bohner M. and Peterson A.,
 {\em Dynamic equations on time scales. An
     introduction with applications},
(Boston, MA: Birkh\"auser Boston Inc.), 2001.

\bibitem{bChe}
	Cheng S.~S., {\em Partial Difference Equations}, Taylor \& Francis, London, 2003.

\bibitem{bCZ}
	Curtain R.~W., Zwart H.~J., {\em An introduction to infinite-dimensional linear systems theory}, Springer, London, 1995.

\bibitem{aDR}
	Dixon M.~J. and Robinson M.~E., A Birth Process Model for Association Football Matches, 
	{\em The Statistician}
    {{\bf 47}(3)(1998), 523--538.}

\bibitem{bEla}
Elaydi S., {\em An Introduction to Difference Equations}, Springer, 2005.

\bibitem{bEva}
Evans L.~C., {\em Partial Differential Equations}, Second Edition, American Mathematical Society, 2010.

\bibitem{bGha}
Ghahramani S., {\em Fundamentals of Probability with Stochastic Processes}, Prentice Hall, 2005.

\bibitem{aHil}
       {\rm Hilger S., }
       {\rm Analysis on measure chains - A unified approach to continuous and discrete calculus, }
       {\em Results Math.}, \textbf{18}(1990) 18--56.

\bibitem{bKP}
Kelley W. G. and Peterson A. C., {\em Difference equations (An introduction with applications)}, Harcourt/Academic Press, 2001.

\bibitem{bLev}
LeVeque R. J., {\em Finite Volume Methods for Hyperbolic Problems (Cambridge Texts in Applied Mathematics)}, Cambridge University Press, 2002.

\bibitem{aMB}
Mozyrska D. and Bartosiewicz Z., Observability of a class of linear dynamic infinite
systems on time scales, {\em Proc. Estonian Acad. Sci. Phys. Math.}, {\bfseries 56}(4)(2007), 347-–358.


\bibitem{bPan}
Panjer H.~H., {\em Operational Risk. Modelling Analytics}, John Wiley \& Sons, 2006.

\bibitem{aRei}
Reid W.~T., Properties of Solutions of an Infinite System of Ordinary Linear Differential Equations of
the First Order with Auxiliary Boundary Conditions, {\em Transactions of the American Mathematical Society}, {\bfseries 32} (2)(1930), 284--318.


\bibitem{aRot}
Rothe E., Zweidimensionale parabolische randwertaufgaben als grenzfall eindimensionaler randwertaufgaben, {\em Mathematische Annalen}, {\bfseries 102} (1930), 650--670.

\bibitem{a_maxpripde}
Stehl\'{\i}k P., Maximum principles for elliptic dynamic equations, {\em Math. Comp. Modelling}, {\bfseries 51} (2010), 1193--1201.


\end{thebibliography}
\end{document}